\documentclass[a4paper, 12pt]{article}
\usepackage{epsfig,amsmath,amssymb,latexsym} 
\usepackage{amscd }
\usepackage{amsfonts}
\usepackage{graphicx}

\newtheorem{theorem}{Theorem}[section]

\newtheorem{lemma}{Lemma}[section]

\newenvironment{proof}[1][Proof]{\par\noindent\textbf{#1.} }{\ \rule{0.5em}{0.5em}}
\newcommand{\N}{\mathbb{N}}

\numberwithin{equation}{section}

\def\VV{{\cal V}}

\begin{document}

\begin{center}
{\huge \textbf{CLT for the proportion of infected individuals for an epidemic model on a complete graph}}\\[0pt]
\vspace{1cm} {\large F. Machado, H. Mashurian and H. Matzinger}
\end{center}

\begin{abstract}We prove a Central Limit Theorem for the
proportion of infected individuals for an epidemic model by
dealing with a discrete time system of simple random walks on a
complete graph with $ n $ vertices. Each random walk makes a role
of a virus. Individuals are all connected as vertices in a complete
graph. A virus duplicates each time it hits a susceptible individual, 
dying as soon as it hits an already infected individual. The
process stops as soon as there is no more viruses.  This model is
closely related to some epidemiologial models like those for virus
dissemination in a computer network.
\end{abstract}
\vspace{0.5cm}

\section{Introduction}

We prove a Central Limit Theorem for the proportion of infected
individuals for an epidemic model. We consider a discrete time 
system of simple random walks on $ K_n $, the $ n $-complete graph, 
a graph with vertex set $\VV = \{1,2, \dots, n\}$ and each pair of vertices
linked by an edge.

This model, also known as frog model, 
has been mostly considered on infinite graphs, in particular hypercubic 
lattices and homogeneous trees, for which results as shape theorem 
and phase transition have been proved. See for 
instance~\cite{PT},~\cite{IUB},~\cite{sofa},~\cite{KPV},~\cite{Serguei},~\cite{RV},
~\cite{TW} and the references therein. A comprehensive introduction on random walks on 
finite and infinite graphs can be found in~\cite{AF}. 

In this paper we deal with a discrete time process on $K_n$ 
evolving as follows. At time zero there is one inactive particle at each vertex of $ K_n.$
A particle is chosen to become active and by its turn that active
particle chooses a vertex to jump at, also activating the particle sitting there.
As at each time just one active particle makes a displacement, one active particle is
uniformely choosen to make its move. From that time on, each active particle perform a random
walk on the vertices of $ K_n $, activating all inactive particles
it meets along its way. Each active particle lives while it
chooses vertices with an inactive particle on it, dying at the
first time it chooses to jump on a vertex which has been visited
before by some active particle. The process continues until there
are no more active particles. 

Considering
\[ V_t =  \hbox{ the number of vertices visited by the process up to time } t, \]
we denote by $V_{\infty} = \lim_{t \to \infty} V_t,$ the number
of vertices which have been visited by active particles when the
process comes to an end. We investigate the asymptotic
distribution of the random variable $V_{\infty}$. The main result
of this paper (Theorem~\ref{maintheorem}) shows that properly
re-scaled, $V_{\infty}$ converges in distribution to a normal
random variable.

Let us formally define the model whose dinamic takes place on $ K_n $. 
First we define $A_t$, $D_t$ and $I_t$ as the number of active particles 
at time $t$, the number of vertices whose original particles have already died up
to time $t$ and the number of particles still inactive at time
$t$, respectively. In this sense, $ V_t = A_t + D_t $ and $ A_t + D_t + I_t = n,$ for all 
discrete time $t.$ Note that $\{(A_t, D_t, I_t)\}_{t \geq 0}$ is a Markov chain 
going

\begin{equation}
\label{E: mctp}
\hbox{from } (a, n-(i+a), i) \left\{\begin{array}{ll}
\hbox{ to } (a+1, n-(i+a), i-1) & \mbox{w.p. $\frac{i}{n}$,} \\ \hbox{ or } \\
\hbox{ to } (a-1, n-(i+a)+1, i) & \mbox{w.p. $\frac{n-i}{n}$,}
\end{array}\right.
\end{equation}
for discrete values of $a \in \{1,2,\dots, n\}$ and $i \in \{0,1,2, \dots, n-1\}.$ 
The chain starts from $A_0=1, D_0=0$ and $V_0 = n-1$ and comes to an end 
as soon as, for some discrete time $t$, $A_t=0.$
Besides, let $\{S_t\}_{t \ge 0}$ denote a set of independent uniformly distributed random
variables on $\VV$, the set of vertices of $K_n$. At each time $ t
$ one active particle (also uniformly chosen among the $A_{t-1}$
active particles), choose the vertex $S_t$ to jump to. It meets
and activates a still inactive particle if and only if $ S_t
\not\in \{S_1, \dots, S_{t-1} \}.$ In this case $ A_{t}=A_{t-1}+1
$. Otherwise that active particle dies, then $ A_{t}=A_{t-1}-1.$
Observe that $ A_{\infty} := \lim_{t \to \infty} A_t = 0 .$ For
simulations and mean field analysis see~\cite{ALMM}

Let $q$ be the only non-zero solution to the equation
$$2p=-\ln(1-p)$$
in $[0,1[$. (See also lemma \ref{map}.)
Let $\mu_r$ be equal to
\[ \mu_r:=2-\frac{1}{1-q}. \]
Finally let $\sigma$ be equal to
\[ \sigma:=\frac{\sqrt{\int_0^q\frac{x}{(1-x)^2}dx}}{\mu_r}=\frac{\sqrt{\frac{q-2q^2}{q-1}}}{\mu_r}. \]
We are now ready to formulate the main theorem of this paper
\begin{theorem}
\label{maintheorem}

We have that
$$\frac{V_{\infty}-qn}{\sigma\sqrt{n}}\rightarrow \mathcal{N}(0,1)$$
as $n$ goes to infinity, where $\rightarrow$ means convergence  in
law.
\end{theorem}

This model can be viewed as an oriented dependent long range
percolation model once one consider the analogous setup on an
infinite connected graph. The main difficulty in answering the
classical questions related to phase transition and shape theorem
in this setup is that the classical coupling techniques cannot be
applied, besides both FKG and BK inequalities fail. In~\cite{CQR}
authors construc a very interesting renewal structure leading to a
definition of regeneration times for which tail estimates are
performed.

Another possible approach and source of interest is to see this
model as an option for modelling the spread of a disease in a
population or spread of viruses in a computer network. Following
the setup we use in this paper the virus duplicates any time it
infects a susceptible individual. Once that happens the individual
becomes immune. The virus dies the first time it tries to infect a
immune individual. The population here is considered finite and
have full contact as every individual can be contacted directly by
any other individual. The  main question we investigate in this
paper corresponds to determine the distribution of the percentage
of the population which escaped from the disease remaining not
infected (but still susceptible) after all the virus are dead. 
For simulations and mean field analysis of this model see~\cite{ALMM}.

\section{Main Ideas}

Let us define $T(s)$, the time it takes for the process to reach $s$ visited
vertices. So, consistently with the process definition, $T(0)=1.$ For $s \in \{2, \dots,n\},$
let

\[ T(s) = \min\{t \in \N: V_t=s\} \]

\noindent
and

\[\rho = \min\{t: A_t = 0\}. \]

Observe that $ A_{\infty} := \lim_{t \to \infty} A_t = A_{\rho} $. From~(\ref{E: mctp}), 
note also that when there are $s=n-i$ visited vertices, each active particle which jumps has
a probability of $s/n$ to die and a probability of $(n-s)/n$ to hit an
inactive particle.

For all $ s $ such that the process has reached the level of $ s $
visited vertices, we define ${\bar X}_s$ as the time the process
spent at that level. Besides ${\bar X}_s$ can also be seen as the
(random) number of active particles which have to jump in so that
the number of visited vertices either goes from $s$ to $s+1$ or
the process finishes.

%
%
%

For the number of visited vertices to go from $s$ to
$s+1$, we need one additional unvisited vertex to be chosen.
Hence,

\[ {\bar X}_s = \left\{\begin{array}{ll}
1 & \mbox{w.p. $(\frac{n-s}{n})$} \\
\cdots & \\
k & \mbox{w.p. $(\frac{s}{n})^{(k-1)}(\frac{n-s}{n})$} \\
\cdots & \\
A_{T(s)}-1 & \mbox{w.p. $(\frac{s}{n})^{(A_{T(s)}-2)}(\frac{n-s}{n})$} \\ \\
A_{T(s)} & \mbox{w.p. $(\frac{s}{n})^{(A_{T(s)}-1)}$}
\end{array}\right.\]

In other words

\[ {\bar X}_s \sim \min\{{\cal G}(\frac{n-s}{n}), A_{T(s)}\} \]

\noindent
where ${\cal G}$ stands for the geometric probability distribution. 

Observe that for realizations of the process such that ${\bar X}_s=A_{T(s)},$ either
the process stops at time $T(s)+A_{T(s)}$ and $T(s+1) = \infty$ or $T(s+1) = T(s) + A_{T(s)}$.


Going from $s$ to $s+1$ visited vertices, the change in the amount
of active particles is designated by ${\bar Y}_s$. For $s$ such that $T(s+1)<\infty$ we define

\[ {\bar Y}_s := A_{T(s+1)}-A_{T(s)} = 2-{\bar X}_s.\]

So we have that

\[ A_{T(s)} = \sum_{i=1}^{s-1} {\bar Y}_i \]

Besides $ {\bar Y}_{V_{\infty}}=-A_{T(V_{\infty})}. $ Note that the
variables ${\bar Y}_1, {\bar Y}_2, \dots, {\bar Y}_i$ are independents
on the event $\{i \le V_{\infty}-1\}$. They are not identically distributed.
 
We now make up an approximation for the model by considering for $s=1,2,\dots$
a sequence of independent $X _s \sim {\cal
G}(\frac{n-s}{n})$ and $ Y_s = 2 - X_s.$
Moreover we consider

\[ W_s = \sum_{i=1}^s Y_i \]

\[\tau:=\min\{s : W_s\leq 1\}.\]

Observe that on the event $\{i \le V_{\infty}-1\}$ it is possible to 
make a coupling such that $(X_i = {\bar X}_i) = 0 a.s.$ and from this 
we have that $ \rho = \tau - 1.$ So, for what comes next we are interested in the
random variable $\tau.$ We show that $\tau$ has expectation of
order $n$, standard deviation of order $n$ and when re-scaled
properly converges to a normal variable.

Let $\mu_s:=E[Y_s]$, so that
\begin{equation}
\label{mu}\mu_s=2-\frac{1}{1-\frac{s}{n}}.
\end{equation}
Note that up to $s<n/2$ we have $\mu_s>0$. On the other hand for
$s>n/2$ we find $\mu_s<0$. This means that about up to $s=n/2$ the
random map $s\mapsto W_s$ increases and after $s=n/2$ it
decreases.

Let $w_s:=E[W_s]$ and let $W_s^*:=W_s-w_s$, whilst
$Y^*_i=Y_i-E[Y_i]$. By these definitions, we get

\[ W^*_s=Y^*_1+Y_2^*+\dots+Y^*_s. \]

The variables $Y^*_1$, $Y^*_2$, ... are independent. Let $c<1$ be
any constant not depending on $n$. Then for $s\leq cn$ the
variables $Y_i$ with $i\leq s$ are stochastically uniformly
bounded by a geometric variable. Hence, $W^*_s$ is typically of
order $\sqrt{s}$ when $s\leq cn$. On the other hand, $s\rightarrow
w_s$ takes on values  which are of order $n$.  Hence, ``the main
shape'' of $s\rightarrow W_s$ is ``determined'' by $s\rightarrow
w_s$ whilst $W_s^*$ only represents
a smaller fluctuation.\\
We have for $s<n$,
\begin{equation}
\label{wssum}w_s=\sum_{i=1}^s\left(  2-\frac{1}{1-\frac{i}{n}}\right),
\end{equation}
which implies that,

$$
w_s\approx n\int_0^{s/n} 2-\frac{1}{1-x}dx.$$
The integral in the expression on the right side of the above approximation,
is equal
$$2s/n+\ln (1-s/n).$$
The next lemma gives the precision of our approximation for $w_s$.

\begin{lemma}
\label{ws} For all $n$ and all $s<n$, we have:
\begin{equation}
\label{boundwsn}
\left|w_s-n\;\left[2(s/n)+\ln(1-(s/n)\right]\;\right|\leq
3+\frac{1}{1-(s/n)}.
\end{equation}
\end{lemma}
\begin{proof}
Let $f$ denote a decreasing function on the interval $[a,b]$. Note
that we have
$$\frac{1}{n}\sum_{i=1}^nf\left(a+(b-a)\frac{i}{n}\right)
\leq \int_a^b f(y)dy\leq \frac{1}{n}
\sum_{i=0}^{n-1}f\left(a+(b-a)\frac{i}{n}\right)$$ and hence
\begin{equation}
\label{integralapprox}
\left|\frac{1}{n}\sum_{i=1}^nf\left(a+(b-a)\frac{i}{n}\right) -
\int_a^b f(y)dy \right|\leq \frac{1}{n}(|f(a)|+|f(b)|).
\end{equation}
The last inequality also holds for increasing functions. Note that
the map $x\mapsto 2-\frac{1}{1-x}$ is everywhere monotone on
$[0,1]$. Hence we can apply to it inequality (\ref{integralapprox})
and find
\begin{equation}
\label{integralapprox2} \left|\sum_{i=1}^s\left(
2-\frac{1}{1-\frac{i}{n}}\right)-
n\int_0^{s/n}2-\frac{1}{1-y}dy\right|\leq 3+\frac{1}{1-\frac{s}{n}}.
\end{equation}
The integral in the expression above can be calculated explicitly:
$$\int_0^{s/n}2-\frac{1}{1-y}dy=2(s/n)+\ln(1-(s/n)).$$
Plugging the expression into
 inequality (\ref{integralapprox2}) yields the desired result.
\end{proof}\\[3mm]
We will see that we only need to consider values of $s$ for which
$s\leq cn$ where $c<1$ is a constant not depending on $n$. Hence the
bound on the right side
of (\ref{boundwsn}) can be treated as a constant bound.\\
The main result in this paper is concerned with finding the (random)
zero $\tau$ of the map $W_s$. In the next lemma, we start by
investigating the zeros of the map $p\rightarrow 2p+\ln(1-p)$, which
is our first approximation of $W_s$.

\begin{lemma}
\label{map}
The map
$$p\mapsto 2p+\ln(1-p)\;\;\;;\;\;\;[0,1[\rightarrow
\mathbb{R}$$ has only one zero $q\in ]0,1[$. Furthermore
\begin{equation}
\label{boundmatlab}
0.796<q<0.798
\end{equation}
\end{lemma}
\begin{proof} The derivative of our map is
$2-1/(1-p)$. It is strictly positive for $p\in [0;1/2[$. So our
map $h(p):=2p+\ln (1-p)$ first increases from the value $h(0)=0$
to the positive value $h(1/2)=1-\ln 2>0$. After then the
derivative of $h(p)$ is strictly negative. Since $h(1/2)>0$ and
$h(1)=-\infty$, we infer that there is only one zero of the map
$h(p)$ in $]0;1[$. The bounds (\ref{boundmatlab}) were obtained by
numeric approximation from above and below.
\end{proof}\\[3mm]

Let $r$ be equal to
$$r:=nq.$$ Due to lemma \ref{ws}, we have that
$w_r$ is close to zero up to a constant.  In other words, $r$ is
approximatively equal to the zero of the map $s\mapsto w_s$. By
definition, $W_s$ is equal to $w_s+W^*_s$, where typically $w_s$
takes on values of order $n$ and $W^*_s$ takes on values of order
$\sqrt{n}$. This implies that the zero  of $s\mapsto W_s$ is equal
to the zero of $w_s$ plus/minus a term of order $\sqrt{n}$. Hence,
the stopping time $\tau$ is typically equal to $r=nq$ plus a random
term with standard deviation of
 order $\sqrt{n}$.\\
How big is the standard deviation of $\tau$? For this,
let us quickly look at another, related problem:
assume that the variables
$\tilde{Y}_1,\tilde{Y}_2,\ldots$ are i.i.d. variables with
finite second moment and $E[\tilde{Y}_1]<0$. Let $K>0$ be a large number,
and let $\tilde{\tau}$ be
$$\tilde{\tau}:=\min\{s| K+\tilde{Y}_1+\tilde{Y}_2+\ldots+
\tilde{Y}_s<0\}.$$
We find that $\tilde{\tau}$ takes typically values which are about
equal to $K/|E[\tilde{Y}_1]|$ with a fluctuation of order $\sqrt{K}$.
(The proof is identical to the proof
 of the Law of Large Numbers and  Central Limit
Theorem for Renewal Processes). \\
In our case, the variables $Y_i$ are not i.i.d but only
independent. However, the variables with $s$ close to
$r=nq$ have all about the same distribution that is geometric
with expectation $\mu_r$.
Note that
$$\mu_r=\mu_{qn}=2-\frac{1}{1-q},$$
is a number not depending on $n$.\\
 We saw that
up to a constant factor, $w_r$ is approximately
equal to zero and hence we have $W_r^*\approx W_r$. Assume that
$W_r^*>0$. Then, for $W_s$ to become zero after $s=r$, we need about
$$-\frac{W_r}{\mu_r}\approx -\frac{W_r^*}{\mu_r}$$
additional ``steps'', i.e. additional variables $Y_s$. (The argument
goes like the argument presented above for the variables $\tilde{Y}_i$).
This yields the approximation
\begin{equation}
\label{approxtau}
\tau\approx r-\frac{W_r^*}{\mu_r}=
qn-\frac{Y^*_1+\ldots+Y^*_{qn}}{\mu_r}.
\end{equation}
The above approximation is typically precise up to a term of order
$n^{1/4}$. This will be proven by introducing some events $B^n_0$,
$B_1^n$ and $B_2^n$ and showing that they when they hold (lemma
\ref{lemmacombinatorics}) then the error in the last approximation
above is of order $n^{1/4}$. In the last section, we prove that
the events $B^n_0$, $B^n_1$ and $B^n_2$ have their probabilities
going to one when $n$ goes to infinity. The expression on the
right side of approximation (\ref{approxtau}) gives the asymptotic
behavior of the standard deviation of $\tau$. The reason is that
the term
\begin{equation}
\label{term}\frac{Y^*_1+\ldots+Y^*_{qn}}{\mu_r}
\end{equation}
 has a standard deviation of order $\sqrt{n}$, whilst the error term
of the approximation (\ref{approxtau}) is of order $n^{1/4}$.
Hence, the standard deviation of (\ref{term}) is asymptotically
equal to the standard deviation of $\tau$ up to a much smaller
error term. Let us calculate the variance of the expression
(\ref{term}). We have that the variables $Y^*_i$ are re-centered
geometric variables with parameter $(n-i)/n$. The variance of
$Y^*_i$ is thus
$$\frac{i/n}{(1-i/n)^2} $$
Hence we find that the variance of the sum (\ref{term}) is equal to
\begin{equation}
\label{variance}
\frac{1}{\mu_r^2}\sum_{i=1}^{qn}VAR[Y^*_i]=
\frac{1}{\mu_r^2}\sum_{i=1}^{qn}\frac{i/n}{(1-i/n)^2}
\end{equation}
The sum in the above expression can be approximated by
an integral. This is the content of the next lemma:
\begin{lemma}
\label{lemmavar}
We have for all $n$ and all $q<1$ that
\begin{equation}
\label{inqn}
\left|\sum_{i=1}^{qn}\frac{i/n}{(1-i/n)^2}
-n\;\int_0^q\frac{x}{(1-x)^2}dx
\right|\leq \frac{q}{(1-q)^2}.
\end{equation}
\end{lemma}
\begin{proof}
Let $h(x):=x/(1-x)^2$. We find that the derivative is equal to
$$h'(x)=\frac{1+x}{(1-x)^3}$$ which is positive for all $x\in[0,1[$.
Hence, inequality (\ref{integralapprox}) can be applied and we find
that inequality (\ref{inqn}) holds.
\end{proof}\\[3mm]
The last lemma above implies that the standard deviation of
(\ref{term}) is approximately equal to $\sigma \sqrt{n}$, where
$$\sigma:=\frac{\sqrt{\int_0^q\frac{x}{(1-x)^2}dx}}{\mu_r}.$$
Note that this is exactly the re-scaling factor used in our main
theorem \ref{maintheorem}!

\section{Combinatorics}
The first event $B^n_0$ is the event that the first $n^{1/4}$
active random walks which jump in, do not get killed:
$$B^n_0:=\{\forall i,j\leq n^{1/4},\; {\rm with}\;i\neq j\;{\rm we\; have},
\;S_i\neq S_j\}.$$ The next event $B^n_1$ is the event that the
approximation of $W_s$ by $w_s$ does not exceed the size $\ln s
\sqrt{s}$. More precisely, $B^n_1$ is the event that for all $s$
with $n^{1/4}\leq s\leq qn$, we have
$$|W^*_s|\leq \ln s \sqrt{s}.$$
The next event $B^n_2$ says that for all $i$ such that
$0\leq i\leq (\ln n)^2\sqrt{n}$ we have
that
$$\left|Y_{qn+1}+Y_{qn+2}+\ldots+Y_{qn+i}-i\mu_r\right|
\leq (\ln n)^3\cdot n^{1/4}$$
and
$$\left|Y_{qn-1}+Y_{qn-2}+\ldots+Y_{qn-i}-i\mu_r\right|
\leq (\ln n)^3\cdot n^{1/4}$$
Next comes our main combinatorial lemma
\begin{lemma}
\label{lemmacombinatorics}
Assume that $B^n_0$, $B^n_1$ and $B^n_2$ all hold,
then we have
$$\left|
\tau- \left(qn-\frac{Y_1^*+Y_2^*+\ldots+Y_{qn}^*}{\mu_r}\right)\right|
\leq 2(\ln n)^3\cdot n^{1/4}
$$
\end{lemma}
\begin{proof}
First, note that when $B^n_0$ holds, then $\tau $ is not in the
interval $[0,n^{1/4}]$. Second, according to lemma
\ref{lemmawbigger}, we have for all $s$ contained in the interval
\begin{equation}
\label{interval}[n^{1/4},nq-(\ln n)^2\sqrt{n}]
\end{equation}
that
$w_s>\ln s\sqrt{s}$. Hence, when $B^n_1$ holds, and
since by definition $W_s=W^*_s+w_s$, we get that
$\tau$ is not in the interval (\ref{interval}).\\
Let $s\mapsto f(s)$ be the (random) linear map
$$f(s)=W_{qn}+(s-qn)\mu_r.$$
Note that the map $f$ has a zero
at
$$qn-\frac{W_{qn}}{\mu_r}.$$
(Note that $\mu_r$ is negative.)
Let $f^+$, resp. $f^-$ be the linear map
$f+(\ln n)^3\cdot n^{1/4}$, resp.
$f-(\ln n)^3\cdot n^{1/4}$. The zero of
$f^+$, resp. $f^-$ is at
$$qn-\frac{W_{qn}+(\ln n)^3\cdot n^{1/4}}{\mu_r},$$
resp.
$$qn-\frac{W_{qn}-(\ln n)^3\cdot n^{1/4}}{\mu_r}$$
Let $I$ be the interval
$$I:=[qn-(\ln n)^2\sqrt{n},qn+(\ln n)^2\sqrt{n}]$$
and let $J$ be the interval
$$J:=[qn-\frac{W_{qn}-(\ln n)^3\cdot n^{1/4}}{\mu_r},
qn-\frac{W_{qn}+(\ln n)^3\cdot n^{1/4}}{\mu_r}].$$
Note that when the event
$B^n_1$ holds, then
\begin{equation}
\label{1}
|W^*_{qn}|\leq \ln n \sqrt{n}.
\end{equation}
By definition
\begin{equation}
\label{2}
W_{qn}=W^*_{qn}+w_{qn}.
\end{equation}
But by equality (\ref{wssum}) and by lemma \ref{ws}, we have for the
constant $k:=3+1/(1-q)$
\begin{equation}
\label{wqn}\left|w_{qn}-n\int_0^{q} 2-\frac{1}{1-x}dx\right|\leq k.
\end{equation}
By definition of $q$, we have
$$\int_0^{q} 2-\frac{1}{1-x}dx=0,$$
so that with inequality (\ref{wqn}), we obtain
\begin{equation}
\label{wqnleqk}
|w_{qn}|\leq k.
\end{equation}
The last inequality together with (\ref{1}) and (\ref{2}) implies
\begin{equation}
\label{Wqn}
|W_{qn}|\leq k+\ln n \sqrt{n}.
\end{equation}
Using inequality (\ref{Wqn}), we obtain that for $n$ large enough
$$J\subset I.$$
Now, when the event $B^n_2$ holds, then in the interval $I$ we have
that $W_s$ is between $f^-$ and $f^+$, that is $f^-(s)\leq W_s\leq
f^+(s)$ for $s\in I$. Hence, in the interval $I$, the map $s\mapsto
W_s$ has its zero between the zeros of $f^-$ and $f^+$. More
precisely, this means that $W_s$ has a zero somewhere in the
interval $I$ and furthermore we have that all zeros of $W_s$ in the
interval $I$ are located in
$J$.\\
We can now summarize what we found so far: when $B^n_0$, $B^n_1$ and
$B^2_n$ all hold, then the map $s\mapsto W_s$ has no zero before the
interval $I$, but within $I$ all the zeros are located in the
subinterval $J$. Hence, $\tau\in J$ which implies
$$\left|\tau-\left(qn-\frac{W_{qn}}{\mu_r}\right)\right|\leq
-(\ln n)^3\cdot n^{1/4}/\mu_r
$$
Using the last equation together with (\ref{wqnleqk}) and (\ref{2}),
we find
$$\left|\tau-\left(qn-\frac{W^*_{qn}}{\mu_r}\right)\right|\leq
-(\ln n)^3\cdot n^{1/4}/\mu_r+k
$$
Note that for $n$ large enough, the right side of the last
inequality is smaller than $2(\ln n)^3\cdot n^{1/4}$. This finishes
proving our lemma
\end{proof}

\begin{lemma}
\label{lemmawbigger} For all $n$ large enough:
every $s$ contained in the interval
\begin{equation}
[\;n^{1/4}\;,\;nq-(\ln n)^2\sqrt{n}\;]
\end{equation}satisfies
\begin{equation}
\label{ws>}w_s>\ln s\sqrt{s}.
\end{equation}
\end{lemma}
\begin{proof}
We consider the three intervals $I_1=[n^{1/4},n/3]$,
$I_2:=[n/3,n/2]$ and $I_3:=[n/2,nq-(\ln n)^2\sqrt{n}]$. We are going
to prove that inequality (\ref{ws>}) holds for each one of them. Let
$h$ designate the map $h(x):=2x+\ln(1-x)$. Note that the second
derivative of $h$ is negative everywhere on $I_1$ for $x=s/n$.
Hence, $h'(x)\geq h'(1/3)>0$ for all $x\in [0,1/3]$. Since $h(0)=0$,
the mean value theorem implies that for all $x\in[0,1/3]$, we have
$h(x)\geq x\cdot h'(1/3)$. When $s\in I_1$ then $(s/n)\in[0,1/3]$ so
that
\begin{equation}
\label{hsn} h(s/n)\geq (s/n)\cdot h'(1/3).
\end{equation}
According to inequality (\ref{boundwsn}), we have
$$w_s\geq n h(s/n)-(3+1/(1-(s/n)))$$
and for $s\in I_1$ since $(s/n)\leq 1/3$,
we obtain
$$w_s\geq n h(s/n)-4.5.$$
The last inequality above together with inequality (\ref{hsn}) then
implies
\begin{equation}
\label{heini}w_s\geq s\cdot h'(1/3)-4.5.
\end{equation}
The expression on the right side of the last inequality
above is larger than $\ln s\sqrt{s}$ for $s$ large enough.
However for $s\in I_1$, we have $s\geq n^{1/4}$, so that
for $n$ large enough, $s$ will be large enough and
$$s\cdot h'(1/3)-4.5\geq \ln s \sqrt{s}.$$
From the last inequality above and (\ref{heini}), we have that
inequality (\ref{ws>}) follows.\\
Next we need to prove (\ref{ws>}) for $s$ in $I_2$. Using inequality
(\ref{boundwsn}) together with the fact that $s/n\leq 1/2$ for $s\in
I_2$, we find
\begin{equation}
\label{heini2}w_s\geq n h(s/n)-5.
\end{equation}
When $s\in I_2$ we have that $s/n\in [1/3,1/2]$. But on the interval
$[1/3,1/2]$ the map $h$ is everywhere increasing. Hence for $s\in
I_2$, we have that $h(s/n)\geq h(1/3)$. Plugging the last inequality
into (\ref{heini2}) gives
\begin{equation}
\label{heini3}w_s\geq n h(1/3)-5.
\end{equation}
For $s\in I_2$ we have $s\leq n$. Hence
\begin{equation}
\label{heini4}\ln s\sqrt{s}\leq \ln n\sqrt{n}.
\end{equation}
For $n$ large enough, $\ln n \sqrt{n}$ is less than $n h(1/3)-5$.
From this and inequalities (\ref{heini3}) and
(\ref{heini4}) inequality (\ref{ws>}) follows.\\
Now, it only remains to prove inequality (\ref{ws>}) for $s\in I_3$.
When $s\in I_3$ we have that $s/n\leq q<1$. This together with
inequality (\ref{boundwsn}) yields
\begin{equation}
\label{heini5}w_s\geq n h(s/n)-(3+1/(1-q)).
\end{equation}
When $s\in I_3$, we have that

$$s/n\in\left[0.5,q-\frac{(\ln n)^2}{\sqrt{n}}\right].$$
On the interval on the right side of the last inclusion above the
map $h$ is everywhere decreasing. Hence, for $s/n\in I_3$ we have
\begin{equation}
\label{q}h(s/n)\geq h\left(q-\frac{(\ln n)^2}{\sqrt{n}}\right).
\end{equation}
Note that by definition $h(q)=0$. Furthermore, $h'(q)<0$. Hence,
using the mean value theorem applied to (\ref{q}), we obtain that
for all $n$ large enough
\begin{equation}
\label{q2}h(s/n)\geq -h'(q)\frac{(\ln n)^2}{2\sqrt{n}}.
\end{equation}
The last inequality together with (\ref{heini5}), gives
\begin{equation}
\label{heini7}w_s\geq -h'(q)\sqrt{n}(\ln n)^2+(3-1/(1-q)).
\end{equation}
For $n$ large enough, the right side of the last inequality above is
larger than $\ln n\sqrt n$ which is larger than $\ln s\sqrt{s}$ when
$s\in I_3$. Hence inequality (\ref{ws>}) holds.
\end{proof}

\section{Probabilities}

\begin{lemma}
\label{4} We have that $P(B^n_0)\rightarrow 1$ as $n\rightarrow
\infty$.
\end{lemma}
\begin{proof}Let $B^n_{0i}$ be the event that the $i$-th frog
jumping in does not die. Hence, $B^n_{0i}$ is the event that $NS_i
\neq S_t$ for all $t<i$. For an event $A^n$ we designate by
$A^{nc}$ its complement. We have that
$$B^n_{0}=\bigcap_{i=1}^{n^{1/4}}B^n_{0i}$$
and hence
\begin{equation}
\label{PBnc}
P(B^{nc}_0)\leq \sum_{i=1}^{n^{1/4}}P(B^{nc}_{0i}).
\end{equation}
Now, for $i\leq n^{1/4}$ there are no more than $n^{1/4}$ vertices
and hence the probability for the $i$-th jumping frog to die is
not more than $n^{1/4}/n=n^{-3/4}$. This immediately implies that
$$P(B^{nc}_{0i})\leq \frac{1}{n^{3/4}}.$$
Using the last inequality with inequality (\ref{PBnc}), we find
$$P(B^{nc}_0)\leq \frac{n^{1/4}}{n^{3/4}}=\frac{1}{\sqrt{n}}.$$
This finishes to prove our lemma.
\end{proof}\\[3mm]

\begin{lemma}
There exist two constants $\kappa>0$ and $c>0$ such that for every
geometric variable $X$ with parameter $p$ satisfying

\begin{equation}
\label{conditiononp}
p\in[1-q,1]
\end{equation} and every  $\Delta\in[0,c]$, we have
\begin{equation}
\label{largedeviationforY}
E[e^{(X-(1/p)- \Delta)\cdot \kappa\Delta}]
\leq e^{-0.5\Delta^2\kappa}
\end{equation}
\end{lemma}
\begin{proof}
Let $\kappa$ be equal to
\begin{equation}
\label{kappa}\kappa:= \min\frac{p^3}{1.1(1-p)(0.1+p)},
\end{equation}
where the minimum is taken over all $p\in[1-q,1]$. Note that $p$ is
bounded away from zero and
$\kappa>0$.\\
Let $c_1>0$ be a number such that for all  $\Delta\in[0,c_1]$ we have

\begin{equation}
\label{conditionI}
\ln \left( 1-\frac{\Delta^2\kappa^2(1-p)(0.1+p)}{2p^3}\right)
\geq -1.1\frac{\Delta^2\kappa^2(1-p)(0.1+p)}{2p^3}.
\end{equation}
Such a number $c_1>0$ exists since for all
$s>0$ small enough we have $\ln(1-s)\geq -1.1s$ and
since $\kappa^2(1-p)(0.1+p)/(2p^3)$ admits a uniform finite upper bound
for  $p\in[1-q,1]$.\\
Let $c_2>0$ be a number such that for all $\Delta\in [0,c_2]$
we have
\begin{equation}
\label{conditionII}
e^{\Delta\kappa/p}\leq 1+\Delta \kappa/p+1.1\Delta^2\kappa^2/(2p^2).
\end{equation}
Such a number $c_2>0$ exists since for all $s>0$ small enough
we have $e^s\leq 1+s+1.1s^2/2$.
Let $c=\min\{c_1,c_2\}$. Hence when $\Delta\in[0,c]$ we have that both
conditions (\ref{conditionI}) and (\ref{conditionII}) are satisfied.\\
Now for the geometric variable $X$ with parameter $p$ we have that
$$E[e^{(X-(1/p)-\Delta)t}]
=\sum_{m=1}^{\infty}e^{mt-t/p-\Delta t}(1-p)^{m-1}p.$$
Using the formula $\sum_{m=1}^{\infty}a^m=a/(1-a)$, we find that for $t$ small enough
$$E[e^{(X-(1/p)-\Delta)t}]=pe^{-t/p-t\Delta}\frac{e^t}{1-(1-p)e^t}=
\frac{pe^{t(1-(1/p)-\Delta)}}{1-e^t(1-p)}.$$
For $t=\kappa \Delta$, we obtain
\begin{equation}
\label{EeX}
E[e^{(X-(1/p)-\Delta)\Delta\kappa}]=
\frac{pe^{\Delta\kappa(1-(1/p)-\Delta)}}{1-e^{\Delta\kappa}(1-p)}=
\frac{pe^{-\Delta^2\kappa}}
{e^{\Delta\kappa(1-p)/p}-(1-p)e^{\Delta\kappa/p}}.
\end{equation}
Note that for any $s>0$ we have
$e^s\geq 1+s+s^2/2$. Hence for $s=\Delta\kappa(1-p)/p$ we find
\begin{equation}
\label{esg}
e^{\Delta\kappa(1-p)/p}\geq 1+\Delta\kappa(1-p)/p+
\Delta^2\kappa^2(1-p)^2/2p^2
\end{equation}

Applying  inequalities (\ref{esg}) and (\ref{conditionII}) to the
expression on the right side of inequality (\ref{EeX}), we find
\begin{align*}
\label{EeXB}
E[&e^{(X-(1/p)-\Delta)\Delta\kappa}]\\
&\leq\frac{ pe^{-\Delta^2\kappa}}
{1+\frac{\Delta\kappa(1-p)}{p}+
\frac{\Delta^2\kappa^2(1-p)^2}{2p^2}
-(1-p)-
\frac{(1-p)\Delta\kappa}{p}-\frac{1.1(1-p)\Delta^2\kappa^2}{2p^2}}\\
&=\frac{ pe^{-\Delta^2\kappa}}
{p+
\frac{\Delta^2\kappa^2(1-p)^2}{2p^2}
-\frac{1.1(1-p)\Delta^2\kappa^2}{2p^2}}\\
&=\frac{ e^{-\Delta^2\kappa}}
{1+
\frac{\Delta^2\kappa^2(1-p)^2}{2p^3}
-\frac{1.1(1-p)\Delta^2\kappa^2}{2p^3}}\\
&=\exp(-\Delta^2\kappa)\cdot\exp(-\ln
(1-\Delta^2\kappa^2(1-p)(0.1+p)/2p^3))
\end{align*}
Applying inequality (\ref{conditionI}) to the most right expression
in the last chain of inequalities above we find
\begin{align}
E[e^{(X-(1/p)-\Delta)\Delta\kappa}]
&\leq
e^{-\Delta^2\kappa}\cdot e^{1.1\Delta^2\kappa^2(1-p)(0.1+p)/2p^3}\\
&\leq
e^{-\Delta^2\kappa(1-1.1\kappa(1-p)(0.1+p)/2p^3)}.\label{EeX1p2}
\end{align}
By the definition (\ref{kappa}) of $\kappa$, we have
$$\kappa\leq \frac{p^3}{1.1(1-p)(0.1+p)}$$
and hence
$$1-\frac{1.1\kappa(1-p)(0.1+p)}{2p^3}\geq 0.5.$$
The last inequality above applied to (\ref{EeX1p2}) yields
$$E[e^{(X-(1/p)-\Delta)\Delta\kappa}]\leq
e^{-0.5\Delta^2\kappa}.$$
\end{proof}\\[3mm]
We can prove the same type of inequality as the one in the lemma
above for the variable $-X$. Hence, we assume that there exist $c>0$
and $\kappa>0$ such that for all $p\in[1-q,1]$ we have that
condition (\ref{largedeviationforY}) is satisfied as well as
\begin{equation}
\label{largedeviationforYII}
E[e^{(-X+(1/p)- \Delta)\cdot \kappa\Delta}]
\leq e^{-0.5\Delta^2\kappa}
\end{equation}
where again $X$ is a geometric variable with parameter $p$.
\begin{lemma}
\label{5} We have that $P(B^n_1)\rightarrow 1$ as $n\rightarrow
\infty$.
\end{lemma}
\begin{proof} Let
$B_{11s}$ be the event
$$B_{11s}:=\{Y_1^*+Y^*_2+\ldots+Y^*_s\leq \ln s\sqrt{s}\}$$
and let
$$B_{12s}:=\{-Y_1^*-Y^*_2-\ldots-Y^*_s\leq \ln s\sqrt{s}\}$$
We have that
$$B^n_1=\bigcap_{s=n^{1/4}}^{nq}\left(B_{11s}\cap
 B_{12s}\right)$$
and hence
\begin{equation}
\label{sum}
P(B^{nc}_1)\leq
\sum_{s=n^{1/4}}^{nq}P(B_{11s}^c)+
\sum_{s=n^{1/4}}^{nq}P( B^c_{12s})
\end{equation}
Recall that for every $t>0$ and any variable $Z$ we have
\begin{equation}
\label{fundamental}
P(Z\geq0)\leq E[e^{tZ}].
\end{equation}
We have
$$P(B^c_{11s})= P((Y_1^*-\Delta)+(Y^*_{2}-\Delta)
+\ldots+(Y^*_s-\Delta)> 0)$$ where $\Delta:= \ln s/\sqrt{s}$. Using
inequality (\ref{fundamental}) yields
\begin{equation}
\label{y*2} P(B^c_{11s})\leq E[e^{t((Y_1^*-\Delta)+(Y^*_2-\Delta)
+\ldots+(Y^*_s-\Delta))}]= \prod_{i=1}^sE[e^{t(Y_i^*-\Delta)}]
\end{equation}
But $Y^*_i=-X_i+1/p_i$, where $X_i$ is a geometric variable with
parameter $p_i$, since by definition $Y^*_i=Y_i-E(Y_i)$ and
$Y_i=2-X_i$. Therefore taking $t=\kappa\Delta$ and applying
inequality (\ref{largedeviationforYII}) to (\ref{y*2}), we obtain
\begin{equation}
\label{Bc11s}
P(B^c_{11s})\leq e^{-0.5\kappa\Delta^2s}=s^{-0.5\kappa \ln s}.
\end{equation}
Similarly one can prove
\begin{equation}
\label{Bc12s}
P(B^c_{12s})\leq e^{-0.5\kappa\Delta^2s}=s^{-0.5\kappa \ln s}.
\end{equation}
Applying inequalities (\ref{Bc11s}) and (\ref{Bc12s}) to inequality
(\ref{sum}) finally gives
$$P(B^{nc}_1)\leq
\sum_{s=n^{1/4}}^{nq}2s^{-0.5\kappa \ln s} $$
and hence $P(B^{nc}_1)$ goes to zero as $n\rightarrow\infty$.
\end{proof}\\[3mm]

To prove that the event $B^n_2$ has high probability we first need
the following lemma:
\begin{lemma}\label{5.5}
for all $i$ such that
$0\leq i\leq (\ln n)^2\sqrt{n}$ we have
that
\begin{equation}
\label{leftmu}\left|\mu_{qn+1}+\mu_{qn+2}+\ldots+\mu_{qn+i}-i\mu_r\right|
\leq (\ln n)^5
\end{equation}
and
\begin{equation}
\label{leftmu2}\left|\mu_{qn-1}+\mu_{qn-2}+\ldots+\mu_{qn-i}-i\mu_r\right|
\leq (\ln n)^5\end{equation}

\end{lemma}
\begin{proof}
Let $f$ be the map defined by $f(x):=2-(1/(1-x))$. Note that $f$ is
continuously differentiable in a neighborhood of $x=q$. Hence there
exists $\delta>0$ such that for all $\Delta\in[-\delta,\delta]$, we
have
\begin{equation}
\label{fq} |f(q+\Delta)-f(q)|\leq c\cdot\Delta
\end{equation}
where $c>0$ is a constant not depending on $\Delta$.
Note that when $i$ satisfies $0\leq i\leq (\ln n)^2\sqrt{n}$
then
$$\left|\frac{i}{n}\right|\leq \frac{(\ln n)^2}{\sqrt{n}}.$$
The right side of the last inequality above goes to zero as
$n\rightarrow\infty$ and hence for $n$ large enough it is less than
$\delta$. We assume now that $n$ is large enough so that
$$\left|\frac{i}{n}\right|\leq \delta,$$
from which by (\ref{fq}) we get
$$|f(q+\frac in)-f(q)|\leq c\frac{i}{n}\leq c \frac{(\ln n)^2}{\sqrt{n}}$$
and equivalently
$$|\mu_{qn+i}-\mu_r|\leq c\frac{(\ln n)^2}{\sqrt{n}}.$$
Applying the last inequality above to the expression on the left
side of inequality (\ref{leftmu}) gives
$$
\left|\mu_{qn+1}+\mu_{qn+2}+\ldots+\mu_{qn+i}-i\mu_r\right|
\leq ic\frac{(\ln n)^2}{\sqrt{n}}\leq c(\ln n)^4.
$$
The term on the right side of the last inequality above for $n$
large enough is less than $(\ln n)^5$ which finishes proving
(\ref{leftmu}). In a similar way we prove (\ref{leftmu2}).
\end{proof}

\begin{lemma}
\label{6}
We have that
$P(B^n_2)\rightarrow 1$
as $n\rightarrow \infty$
\end{lemma}
\begin{proof}
Hint: Use lemma \ref{5.5} and the Hoeffding inequality.
\end{proof}\\[3mm]

\noindent{\bf Proof of the main theorem \ref{maintheorem}}
Lemma \ref{lemmacombinatorics} states that
when $B_0^n$, $B^n_1$ and $B^n_2$ all hold then
\begin{equation}
\label{maininequality}
\left|\;
\frac{\tau-qn}{\sigma\sqrt{n}}\;-\;
\frac{-Y^*_1-Y^*_2-\ldots-Y^*_{qn}}
{\sqrt{\int_0^q\frac{x}{(1-x)^2}dx}\sqrt{n}}\;\right|\leq
\frac{2(\ln n)^3}{\sigma n^{1/4}}.
\end{equation}
From lemma \ref{lemmavar} and equality (\ref{variance}) it follows
that the standard deviation of the sum
$$Y^*_1+Y^*_2+\ldots +Y^*_{qn}$$ is equal up to a constant term
to $$\sqrt{\int_0^q\frac{x}{(1-x)^2}dx}\sqrt{n}.$$ Hence by the
Central Limit Theorem for independent but non-identical variables we
have that the re-scaled sum
$$\frac{-Y^*_1-Y^*_2-\ldots-Y^*_{qn}}
{\sqrt{\int_0^q\frac{x}{(1-x)^2}dx}\sqrt{n}}$$ converges weakly to a
Standard Normal variable. From this and from the fact that
inequality (\ref{maininequality}) holds with probability converging
to one when $n$ goes to infinity we get that
$(\tau-qn)/(\sigma\sqrt{n})$ converges weakly to a standard normal.
We also used the fact that the right side of (\ref{maininequality})
goes to zero as $n$ goes to infinity. Inequality
(\ref{maininequality}) holds with probability going to one when $n$
goes to infinity, because the events $B^n_0$, $B^n_1$ and $B^n_2$,
which together imply (\ref{maininequality}), all have their
probabilities going to one as $n$ goes to infinity.

\section*{Acknowledgements}

Authors are thankful to FAPESP and CNPq for the financial support. 
Thanks are also due for the anonymous referees for their careful reading, 
corrections, criticism and suggestions which helped us to improve the paper.


\begin{thebibliography}{99}

\bibitem{AF} {D. Aldous, J. Fill.}
\textit{Reversible Markov Chains and Random Walks on Graphs.}
Available at
http://www.stat.berkeley.edu/users/aldous/RWG\\/book.html.

\bibitem{PT} {O. Alves, F. Machado and S. Popov.}
Phase transition for the frog model. \textit {Electron. J.
Probab.} \textbf {7}, no.~16, 1--25 (2002).

\bibitem{IUB} {E. Lebensztayn, F. Machado and S. Popov.}
An improved upper bound for the critical probability of the frog
model on homogeneous trees. \textit {J. Statist. Phys.}
\textbf{119}, no.~1-2, 331--345 (2005).

\bibitem{sofa} {O. Alves, F. Machado and S. Popov.}
The shape theorem for the frog model. {\it Ann. Appl. Probab.}
{\bf 12} (2), 534--547 (2002).

\bibitem{ALMM} {O. Alves, E. Lebensztayn, F. Machado,
M. Zuluaga.} Random walk systems on complete graphs \textit {BSBM}
\textbf {X}, no. 11, 22 p (2006).

\bibitem{CQR} {F. Comets, J. Quastel, A. Ram\'{\i}irez.} Fluctuations of
the front in a stochastic combustion model \textit{Annales de
LÍnstitut Henri Poincare} (B) Probability and Statistics
\textbf{43}, no. 2, 147-162 (2007).

\bibitem{KPV} {I. Kurkova, S. Popov and M. Vachkovskaia.}
On infection spreading and competition between independent random
walks. \textit {Electron. J. Probab.} \textbf {9}, no. 11, 1 - 22
(2004).

\bibitem{Serguei} {S. Popov.}
Frogs in random environment. {\it J. Statist. Phys.} {\bf 102}
(1/2), 191--201  (2001).

\bibitem{RV} {A.F.Ram\'{\i}rez and V. Sidoravicius.}
Asymptotic behavior of a stochastic combustion growth process.
\textit{J. Eur. Math. Soc.} (JEMS) 6,  no. 3, 293--334 (2004).

\bibitem{TW} {A. Telcs and N. Wormald.}
Branching and tree indexed random walks on fractals. \textit {J.
Appl. Probab.} \textbf {36}, 999--1011 (1999).

\end{thebibliography}

\end{document}